\documentclass[12pt,leqno]{amsart}
\usepackage{mathrsfs,dsfont}
\usepackage{amsmath,amstext,amsthm,amssymb}
\usepackage{charter}

\newtheorem{Theorem}{Theorem}[section]

\newtheorem{Lemma}[Theorem]{Lemma}

\theoremstyle{definition}

\newtheorem{Example}[Theorem]{Example}
\newtheorem{Remark}[Theorem]{Remark}

\DeclareMathOperator{\codim}{codim}

\newcommand{\comment}[1]{}

\begin{document}

\title[Approximation of positive closed currents on K\"ahler manifolds]{On the approximation of positive closed currents on compact K\"ahler manifolds}
\author{Dan Coman \and George Marinescu}
\thanks{D. Coman is partially supported by the NSF Grant DMS-1300157}
\thanks{G. Marinescu is partially supported by SFB TR 12}
\date{January 24, 2013}
\address{D. Coman: Department of Mathematics, Syracuse University, Syracuse, NY 13244-1150, USA}\email{dcoman@syr.edu}
\address{G. Marinescu: Universit\"at zu K\"oln, Mathematisches Institut, Weyertal 86-90, 50931 K\"oln, Germany
    \& Institute of Mathematics `Simion Stoilow', Romanian Academy,
Bucharest, Romania}\email{gmarines@math.uni-koeln.de}

\pagestyle{myheadings}

\begin{abstract} Let $L$ be a holomorphic line bundle over a compact K\"ahler manifold $X$ endowed with a singular Hermitian metric $h$ with curvature current $c_1(L,h)\geq0$. In certain cases when the wedge product $c_1(L,h)^k$ is a well defined current for some positive integer $k\leq\dim X$, we prove that $c_1(L,h)^k$ can be approximated by averages of currents of integration over the common zero sets of $k$-tuples of holomorphic sections over $X$ of the high powers $L^p:=L^{\otimes p}$.

\par In the second part of the paper we study the convergence of the Fubini-Study currents and the equidistribution of zeros of $L^2$-holomorphic sections of the adjoint bundles $L^p\otimes K_X$, where $L$ is a holomorphic line bundle over a complex manifold $X$ endowed with a singular Hermitian metric $h$ with positive curvature current. As an application, we obtain an approximation theorem for the current $c_1(L,h)^k$ using currents of integration over the common zero sets of $k$-tuples of sections of $L^p\otimes K_X$.
\end{abstract}

\maketitle

\section{Introduction}

Let $L$ be a holomorphic line bundle over the compact K\"ahler manifold $(X,\Omega)$ and let $n=\dim X$. Given a holomorphic section $\sigma\in H^0(X,L^p)$, where $L^p:=L^{\otimes p}$, we denote by $[\sigma=0]$ the current of integration (with multiplicities) over the analytic hypersurface $\{\sigma=0\}\subset X$. We define $\mathcal A_k(L^p)$, $k\leq n$, to be the space of positive closed currents $R$ of bidegree $(k,k)$ on $X$ of the form
\begin{equation}\label{e:aci}
R=\frac{1}{p^kN}\,\sum_{\ell=1}^N\,[\sigma_{\ell,1}=0]\wedge\ldots\wedge[\sigma_{\ell,k}=0]\,.
\end{equation}
Here $N\in\mathbb N$ and $\sigma_{\ell,j}\in H^0(X,L^p)$ are such that $[\sigma_{\ell,1}=0]\wedge\ldots\wedge[\sigma_{\ell,k}=0]$ is a well defined positive closed current on $X$ in the sense of \cite{D93,FS95}. Recall that in this case the set $\{\sigma_{\ell,i_1}=0\}\cap\ldots\cap\{\sigma_{\ell,i_j}=0\}$ has pure dimension $n-j$ for every $i_1<\ldots<i_j$ in $\{1,\dots,k\}$ and the current $[\sigma_{\ell,1}=0]\wedge\ldots\wedge[\sigma_{\ell,k}=0]$ is equal to the current of integration with multiplicities along the analytic set $\{\sigma_{\ell,1}=0\}\cap\ldots\cap\{\sigma_{\ell,k}=0\}$ \cite[Corollary 2.11, Proposition 2.12]{D93}.
Building on our previous work \cite{CM11} we will prove the following approximation result for wedge products of integral positive closed currents by currents of integration along the common zero sets of $k$-tuples of holomorphic sections of high powers of $L$.

\begin{Theorem}\label{T:approx} Let $(X,\Omega)$ be a compact K\"ahler manifold and $(L,h_0)$ be a holomorphic line bundle on $X$ endowed with a singular Hermitian metric $h_0$ with strictly positive curvature current, i.e. $c_1(L,h_0)\geq c\,\Omega$ on $X$, for some constant $c>0$. Assume that $h_0$ is continuous on $X\setminus\Sigma$, where $\Sigma$ is an analytic subset of $X$. If $k\leq\codim\Sigma$ and $h$ is a singular Hermitian metric on $L$ with curvature current $c_1(L,h)\geq0$ on $X$ and with locally bounded weights on $X\setminus\Sigma$, then there exists a sequence of currents $R_j\in\mathcal A_k(L^{p_j})$, where $p_j\nearrow\infty$, such that $R_j$ converges weakly on $X$ to $c_1(L,h)^k$.
\end{Theorem}

\par The hypothesis on $h$ means the following: if $U_\alpha$ is an open set of $X$ on which there exists a local holomorphic frame $e_\alpha$ of $L$ and $|e_\alpha|_h=e^{-\psi_\alpha}$, then $\psi_\alpha$ is plurisubharmonic (psh) on $U_\alpha$ and locally bounded on $U_\alpha\setminus\Sigma$. Note that in Theorem \ref{T:approx} the line bundle $L$ is big and $X$ is Moishezon \cite{JS93}. Hence $X$ is a projective manifold, since it is assumed to be K\"ahler (see e.\,g.\ \cite[Theorem 2.2.26]{MM07}).

\par The proof relies on the convergence of Fubini-Study currents obtained in \cite[Theorem 5.4]{CM11} . This follows in turn from the weak asymptotic behavior of the Bergman kernel, $\frac{1}{p}\log P_p\to 0$ in $L^1_{loc}$, see \cite[Theorems 1.1, 1.2]{CM11}. For equidistribution results concerning some non-compact manifolds see \cite{DMS}, where the full asymptotic expansion of the Bergman kernel \cite[Theorem 6.1.1]{MM07} \cite[Theorem 3.11]{MM08} is used.

\par In the case that the line bundle $L$ is {\em positive} (i.\,e.\ it admits a smooth Hermitian metric with positive curvature), one has the following theorem in which the approximation is achieved by using currents of integration over common zero sets of $k$-tuples of sections rather than averages of such.

\begin{Theorem}\label{T:approx2} Let $X$ be a compact K\"ahler manifold of dimension $n$ and $L$ be a positive holomorphic line bundle on $X$. Assume that $h$ is a singular Hermitian metric on $L$ with positive curvature current $c_1(L,h)\geq0$ such that the current $c_1(L,h)^k$ is well defined for some $k\leq n$. Then there exist a sequence of integers $p_j\nearrow\infty$ and sections $\sigma_{j,1},\dots,\sigma_{j,k}\in H^0(X,L^{p_j})$ such that $T_j:=[\sigma_{j,1}=0]\wedge\ldots\wedge[\sigma_{j,k}=0]$ are well defined positive closed currents of bidegree $(k,k)$ and $p_j^{-k}\,T_j\to c_1(L,h)^k$ weakly on $X$.
\end{Theorem}

\par Theorem \ref{T:approx2} is a consequence of the equidistribution results of common zeros of $k$-tuples of holomorphic sections \cite{DS06b,ShZ99,ShZ08} and the regularization results of \cite{D92,GZ05,BK07}. The hypothesis on the wedge product current $c_1(L,h)^k$ is that it is well defined locally, in the sense of Bedford and Taylor \cite{BT76,BT82}. We recall briefly its definition at the beginning of Section \ref{S:pfat}.

\par Theorem \ref{T:approx2} for $k=1$ is due to Demailly (\cite[Proposition 9.1]{D93b}, see also \cite[Th\'eor\`eme 1.9]{D82I}). 
In the case when $X$ is a projective homogeneous manifold or when $X\setminus{\rm supp}\,c_1(L,h)$ satisfies a certain convexity property it was shown in \cite[Theorems 0.5, 1.6]{G99} that the approximating sections $\sigma_j$ can be chosen is such a way that $\{\sigma_j=0\}$ converge in the Hausdorff metric to the support of $c_1(L,h)$. Such an approximation theorem with control on the supports was first proved in \cite{DS95} for positive closed currents of bidegree (1,1) on pseudoconvex domains in ${\mathbb C}^n$.

\par When $k=1$, Theorem \ref{T:approx2} holds in fact in the more general case of a {\em big} line bundle with an arbitrary positively curved singular Hermitian metric.

\begin{Theorem}\label{T:approx3} Let $L$ be a big line bundle over the compact K\"ahler manifold $X$ and $h$ be a singular Hermitian metric on $L$ with positive curvature current $c_1(L,h)\geq0$. Then there exists a sequence of sections $\sigma_j\in H^0(X,L^{p_j})$, where $p_j\nearrow\infty$, such that $p_j^{-1}\,[\sigma_j=0]\to c_1(L,h)$ weakly on $X$ as $j\to\infty$.
\end{Theorem}

\par Theorems \ref{T:approx}, \ref{T:approx2} and \ref{T:approx3} are proved in Section \ref{S:pfat}. In Section \ref{S:rmad} we consider the general framework as in \cite{CM11} of a holomorphic line bundle $L$ over a complex manifold $X$ endowed with a positively curved singular Hermitian metric $h$ which is continuous outside a compact analytic set $\Sigma\subset X$, and we work instead with the spaces of $L^2$-holomorphic sections of the adjoint bundles $L^p\otimes K_X$ relative to the metrics induced by $h$ and some positive $(1,1)$ form on $X$. If $\gamma_p$ are the corresponding Fubini-Study currents and $k\leq\codim\Sigma$ we prove that $\frac{1}{p^k}\,\gamma_p^k$ converge weakly on $X$ to $c_1(L,h)^k$. We also study the equidistribution of common zeros of $k$-tuples of $L^2$-holomorphic sections of $L^p\otimes K_X$. We conclude Section \ref{S:rmad} by noting that working with sections of adjoint bundles one can obtain a more general approximation theorem than Theorem \ref{T:approx}, in the sense that the metric $h_0$ is assumed to verify a weaker positivity condition. Let $\mathcal A_k(L^p\otimes K_X)$ be the spaces of positive closed currents of bidegree $(k,k)$ on $X$ defined as in \eqref{e:aci} using sections $\sigma_{\ell,j}\in H^0(X,L^p\otimes K_X)$. Namely, we prove the following:

\begin{Theorem}\label{T:approx4} Let $(X,\Omega)$ be a compact K\"ahler manifold and $(L,h_0)$ be a holomorphic line bundle on $X$ endowed with a singular Hermitian metric $h_0$ with positive curvature current $c_1(L,h_0)\geq0$. Assume that there exists an analytic subset $\Sigma$ of $X$ such that $h_0$ is continuous on $X\setminus\Sigma$ and $c_1(L,h_0)$ is strictly positive on $X\setminus\Sigma$, i.e. $c_1(L,h_0)\geq\varepsilon\,\Omega$ for some continuous function $\varepsilon:X\setminus\Sigma\longrightarrow(0,+\infty)$. If $k\leq\codim\Sigma$ and $h$ is a singular Hermitian metric on $L$ with curvature current $c_1(L,h)\geq0$ on $X$ and with locally bounded weights on $X\setminus\Sigma$, then there exists a sequence of currents $R_j\in\mathcal A_k(L^{p_j}\otimes K_X)$, where $p_j\nearrow\infty$, such that $R_j$ converges weakly on $X$ to $c_1(L,h)^k$.
\end{Theorem}

\section{Proofs of Theorems \ref{T:approx}, \ref{T:approx2} and \ref{T:approx3}}\label{S:pfat}

\par We recall first the inductive definition of complex Monge-Amp\`ere operators due to Bedford and Taylor. Let $T$ be a positive closed current of bidimension $(\ell,\ell)$, $\ell>0$, on an open set $U$ in ${\mathbb C}^n$ and let $|T|$ denote its trace measure. If $u$ is a psh function on $U$ so that $u\in L^1_{loc}(U,|T|)$ we say that the wedge product $dd^cu\wedge T$ is well defined. This is the positive closed current of bidimension $(\ell-1,\ell-1)$ defined by $dd^cu\wedge T=dd^c(uT)$.

\par If $u_1,\dots,u_k$ are psh functions on $U$ we say that $dd^cu_1\wedge\ldots\wedge dd^cu_k$ is well defined if one can define inductively as above all intermediate currents
$$dd^cu_{j_1}\wedge\ldots\wedge dd^cu_{j_\ell}=dd^c(u_{j_1}dd^cu_{j_2}\wedge\ldots\wedge dd^cu_{j_\ell}),\;1\leq j_1<\ldots<j_\ell\leq k.$$
The wedge product is well defined for locally bounded psh functions \cite{BT76,BT82}, for psh functions that are locally bounded outside a compact subset of a pseudoconvex open set $U$, or when the mutual intersection of their unbounded loci is small in a certain sense \cite{Si85,D93, FS95}. If $T_j$, $1\leq j\leq k$, are positive closed currents of bidegree $(1,1)$ on a complex manifold then one defines  $T_1\wedge\ldots\wedge T_k$ by working with their local psh potentials.

\medskip

\par We will need the following lemma:

\begin{Lemma}[{\cite[Lemma 2.1, p.\,182]{Kuipers}}]\label{L:discrete} Let $\mu$ be a positive Radon measure on a compact metric space $K$ with $\mu(K)=1$, let $\mathcal{B}$ be the $\sigma$-algebra of Borel sets in $K$, and let $f:K\longrightarrow{\mathbb C}$ be a given bounded Borel measurable function. Then
$$\lim_{N\to\infty}\frac{1}{N}\,\sum_{\ell=1}^Nf(x_\ell)=\int_Kf\,d\mu$$
for almost all sequences $(x_\ell)\in\prod_{\ell=1}^\infty(K,\mathcal{B},\mu)$.
\end{Lemma}

\begin{proof}[Proof of Theorem \ref{T:approx}] Since $\dim\Sigma\leq n-k$ and $h$ has locally bounded weights on $X\setminus\Sigma$, $c_1(L,h)^k$ is a well defined positive closed current on $X$ by \cite[Corollary 2.11]{D93}. The proof of Theorem \ref{T:approx} is divided in three steps.

\smallskip

\par {\em Step 1.} We introduce here the spaces of $L^2$-holomorphic section of $L^p$ and their Fubini-Study currents that will be used in the proof. Assume for the moment that $h$ is any singular metric on $L$ with positive curvature current $c_1(L,h)\geq0$. Let $H^0_{(2)}(X,L^p)$ be the space of $L^2$-holomorphic sections of $L^p$ relative to the metric $h_p$ induced by $h$ and the volume form  $\Omega^n$ on $X$,
$$H^0_{(2)}(X,L^p)=\Big\{S\in H^0(X,L^p):\,\|S\|_p^2:=\int_X|S|_{h_p}^2\,\Omega^n<\infty\Big\}\,,$$
endowed with the obvious inner product. Let $d_p=\dim H^0_{(2)}(X,L^p)$ and fix an orthonormal basis $\{S^p_i\}_{1\leq i\leq d_p}$ of $H^0_{(2)}(X,L^p)$. We denote by $\gamma_p$ the Fubini-Study current of the space $H^0_{(2)}(X,L^p)$, defined by
\begin{equation}\label{e:gammap}
\gamma_p\mid_{_{U_\alpha}}=\frac{1}{2}\,dd^c\log\left(\sum_{i=1}^{d_p}|s_i^p|^2\right),\;U_\alpha\subset X\,\text{ open}\,,
\end{equation}
where $d^c=\frac{1}{2\pi i}(\partial-\overline\partial)$, $S_i^p=s_i^pe_\alpha^{\otimes p}$, and $e_\alpha$ is a local holomorphic frame for $L$ on $U_\alpha$. Recall that the currents $\gamma_p$ are independent of the choice of basis $\{S_i^p\}$ and we have
$$\frac{1}{p}\,\gamma_p=c_1(L,h)+\frac{1}{2p}\,dd^c\log P_p\,,$$
where $P_p=\sum_{i=1}^{d_p}|S_i^p|_{h_p}^2$ is the Bergman kernel function of the space $H^0_{(2)}(X,L^p)$ (cf.\ \cite[Lemma 3.2]{CM11} or \cite[Lemma 3.8]{CM12}). If $S\in H^0(X,L^p)$ the Lelong-Poincar\'e equation \cite[Theorem\,2.3.3]{MM07} takes the form
$$\frac{1}{p}\,[S=0]=c_1(L,h)+\frac{1}{2p}\,dd^c\log |S|_{h_p}^2\,.$$
Let $\theta$ be a closed smooth (1,1) form representing the first Chern class of $L$. Then any positive closed current $T$ of bidegree (1,1) in the first Chern class of $L$ can be written as $T=\theta+dd^c\varphi$ where $\varphi$ is a $\theta$-psh function on $X$. These are quasiplurisubharmonic (qpsh) functions $\varphi$ on $X$, i.e. functions that can be written locally as a sum of a psh function and a smooth one, such that the current $\theta+dd^c\varphi\geq0$. It follows by the inductive (local) definition of the complex Monge-Amp\`ere operator that if, for some $k\leq n$, the wedge product currents $c_1(L,h)^k$, $\gamma_p^k$, $[\sigma_1=0]\wedge\ldots\wedge[\sigma_k=0]$ are well defined, where $\sigma_i\in H^0(X,L^p)$, then
\begin{equation}\label{e:mass}
\frac{1}{p^k}\int_X\gamma_p^k\wedge\Omega^{n-k}=\frac{1}{p^k}\int_X[\sigma_1=0]\wedge\ldots\wedge[\sigma_k=0]\wedge\Omega^{n-k}=\int_Xc_1(L,h)^k\wedge\Omega^{n-k}\,,
\end{equation}
all being equal to $\int_X\theta^k\wedge\Omega^{n-k}$ (cf.\ \cite{GZ05}).

\par We identify the unit sphere ${\mathcal S}^p$ of $(H^0_{(2)}(X,L^p),\|\cdot\|_{L^2})$ to the unit sphere ${\mathbf S}^{2d_p-1}$ in ${\mathbb C}^{d_p}$ by
$${\mathbf S}^{2d_p-1}\ni a=(a_1,\dots,a_{d_p})\longmapsto S_a=\sum_{i=1}^{d_p}a_iS^p_i\in{\mathcal S}^p,$$
and we let $\lambda_p$ be the probability measure on ${\mathcal S}^p$ induced by the normalized surface measure on ${\mathbf S}^{2d_p-1}$, denoted also by $\lambda_p$ (i.e.\ $\lambda_p({\mathbf S}^{2d_p-1})=1$). If $k\geq2$, we let $\lambda_p^k$ denote the product measure on $({\mathcal S}^p)^k$ determined by $\lambda_p$.

\smallskip

\par {\em Step 2.} We consider now the case when the metric $h$ has strictly positive curvature current and is continuous on $X\setminus\Sigma$. By \cite[Theorem 5.4]{CM11} (see also \cite[Theorems 1.1, 1.2]{CM11}), there exists $p_0$ so that the currents $c_1(L,h)^k$ and $\gamma_p^k$, for $p\geq p_0$, are well defined and $\frac{1}{p^k}\,\gamma_p^k\to c_1(L,h)^k$ weakly on $X$ as $p\to\infty$. Moreover, for each $p\geq p_0$ there exists a Borel set $\mathcal N_p\subset({\mathcal S}^p)^k$ with $\lambda_p^k(\mathcal N_p)=0$ such that
$$[\sigma=0]:=[\sigma_1=0]\wedge\ldots\wedge[\sigma_k=0]\,,\,\text{ where }\sigma=(\sigma_1,\dots,\sigma_k)\in({\mathcal S}^p)^k\setminus\mathcal N_p\,,$$
is a well defined positive closed current of bidegree $(k,k)$ on $X$ and
$$\int_{({\mathcal S}^p)^k}\big\langle[\sigma=0],\theta\big\rangle\,d\lambda_p^k=\langle\gamma_p^k,\theta\rangle\,,$$
for every test form $\theta$ on $X$.

\par Let $\{\theta_m\}_{m\geq1}$ be a dense set of smooth $(n-k,n-k)$ forms on $X$ and let $p\geq p_0$. We define the function
$$f_m:({\mathcal S}^p)^k\to\mathbb{C},\:\:f_m(\sigma)=\left\{\begin{array}{ll}\langle[\sigma=0],\theta_m\rangle\,,\,\text{ if }\sigma\in({\mathcal S}^p)^k\setminus\mathcal N_p\,,\\
0\,,\,\text{ if }\sigma\in\mathcal N_p\,.\end{array}\right.$$
Then $f_m$ is Borel measurable and bounded in view of \eqref{e:mass}. Since $\lambda_p^k(\mathcal N_p)=0$, Lemma \ref{L:discrete}, applied to the finite collection of functions $f_m$, $1\leq m\leq p$, implies that for almost all sequences $(\sigma_\ell)\in\prod_{\ell=1}^\infty(({\mathcal S}^p)^k\setminus\mathcal N_p,\lambda_p^k)$ we have
$$\frac{1}{N}\,\sum_{\ell=1}^Nf_m(\sigma_\ell)=\frac{1}{N}\,\sum_{\ell=1}^N\langle[\sigma_\ell=0],\theta_m\rangle\to\int_{({\mathcal S}^p)^k}f_m\,d\lambda_p^k=\langle\gamma_p^k,\theta_m\rangle\,,$$
as $N\to\infty$, for all $m$, $1\leq m\leq p$.

\par It folows that for each $p\geq p_0$ there exists a current $T_p=\frac{1}{N_p}\,\sum_{\ell=1}^{N_p}[\sigma^p_\ell=0]$, where $N_p\in\mathbb N$, $\sigma^p_\ell\in({\mathcal S}^p)^k\setminus\mathcal N_p$, such that
\begin{equation}\label{e:Rp}
\left|\langle T_p,\theta_m\rangle-\langle\gamma_p^k,\theta_m\rangle\right|<1\;,\;\;1\leq m\leq p\,.
\end{equation}
By \eqref{e:mass} the currents $\frac{1}{p^k}\,\gamma_p^k$ and $R_p:=\frac{1}{p^k}\,T_p\in\mathcal A_k(L^p)$ have uniformly bounded mass. Since $\frac{1}{p^k}\,\gamma_p^k\to c_1(L,h)^k$, we conclude by \eqref{e:Rp} that $R_p\to c_1(L,h)^k$ weakly on $X$ as $p\to\infty$.

\smallskip

\par {\em Step 3.} We treat here the general case when $h$ has locally bounded weights on $X\setminus\Sigma$ and $c_1(L,h)\geq0$. Fix a smooth Hermitian metric $h_s$ on $L$ and let $\theta=c_1(L,h_s)$. By \cite{D90} (see also \cite{GZ05}) the set of positively curved singular metrics on $L$ is in one-to-one correspondence with the set of $\theta$-psh functions on $X$. Let $\varphi$, resp. $\varphi_0$, be the $\theta$-psh function determined by $h$, resp. $h_0$.

\par By hypothesis, $\varphi$ is locally bounded on $X\setminus\Sigma$, $\varphi_0$ is continuous on $X\setminus\Sigma$, and
$$c_1(L,h)=\theta+dd^c\varphi\geq0\,,\;\,c_1(L,h_0)=\theta+dd^c\varphi_0\geq c\,\Omega\,.$$
Subtracting a constant we may assume $\varphi<0$ on $X$. By \cite[Theorem 3.2]{DP04} (see also \cite[Proposition 3.7]{D92}) there exist a sequence of qpsh functions $\varphi_j<0$ decreasing to $\varphi$ on $X$ and a sequence $\varepsilon_j\in(0,c)$, $\varepsilon_j\searrow0$, such that $\varphi_j$ are smooth on $X\setminus\Sigma$ and $\theta+dd^c\varphi_j\geq-\varepsilon_j\Omega$ on $X$. If $\widetilde\varphi_j=(1-\varepsilon_j/c)\varphi_j+(\varepsilon_j/c)\varphi_0$ then
$$\theta+dd^c\widetilde\varphi_j\geq\theta-(1-\varepsilon_j/c)(\theta+\varepsilon_j\Omega)+(\varepsilon_j/c)(-\theta+c\,\Omega)=(\varepsilon_j^2/c)\,\Omega\,,$$
hence $\widetilde\varphi_j$ are $\theta$-psh. Let $h_j$ denote the metric on $L$ determined by $\widetilde\varphi_j$, so $c_1(L,h_j)\geq(\varepsilon_j^2/c)\,\Omega$ on $X$ and $h_j$ is continuous on $X\setminus\Sigma$.

\par We claim that $c_1(L,h_j)^k\to c_1(L,h)^k$ weakly on $X$ as $j\to\infty$. Indeed, let $U\subset X$ be a coordinate ball and write $\theta=dd^c\chi$, $\Omega=dd^c\rho$, where $\chi,\rho$ are smooth functions so that $\chi<0$, $\rho>0$ on $U$. The functions
$$\psi_j=(1-\varepsilon_j/c)(\chi+\varphi_j)+\varepsilon_j\rho\,,\;\,\psi_0=\chi+\varphi_0-c\rho\,,$$
are continuous on $U\setminus\Sigma$ and psh on $U$, since $dd^c\psi_j\geq(\varepsilon_j^2/c)\,\Omega$, $dd^c\psi_0\geq0$. As $\chi<0$, $\rho>0$, $0>\varphi_j\searrow\varphi$, $\varepsilon_j\searrow0$, it follows that $\psi_j\searrow\chi+\varphi$ on $U$. Note that on $U$
$$c_1(L,h_j)=\theta+dd^c\widetilde\varphi_j=dd^c\widetilde\psi_j\,\,\text{ where } \widetilde\psi_j=\psi_j+(\varepsilon_j/c)\psi_0\,.$$
Since $k\leq\codim\Sigma$, it follows by \cite[Corollary 2.11]{D93} that $(dd^c\psi_j)^{k-\ell}\wedge(dd^c\psi_0)^\ell$, $(dd^c(\chi+\varphi))^{k-\ell}\wedge(dd^c\psi_0)^\ell$, $0\leq\ell\leq k$, are well defined positive closed currents on $U$. Moreover, since $\varepsilon_j\to0$ and by \cite[Proposition 2.9]{D93}, $(dd^c\psi_j)^{k-\ell}\wedge(dd^c\psi_0)^\ell\to(dd^c(\chi+\varphi))^{k-\ell}\wedge(dd^c\psi_0)^\ell$, we have
$$c_1(L,h_j)^k=(dd^c\widetilde\psi_j)^k=\sum_{\ell=0}^k\binom{k}{\ell}\frac{\varepsilon_j^\ell}{c^\ell}\,(dd^c\psi_j)^{k-\ell}\wedge(dd^c\psi_0)^\ell\to c_1(L,h)^k\,,$$
weakly on $U$ as $j\to\infty$.

\par Let $\{\theta_m\}_{m\geq1}$ be a dense set of smooth $(n-k,n-k)$ forms on $X$. Applying the result of Step 2 to each metric $h_j$ we obtain a sequence of integers $p_j\nearrow\infty$ and of currents $R_j\in\mathcal A_k(L^{p_j})$ such that
$$\left|\big\langle R_j,\theta_m\big\rangle-\big\langle c_1(L,h_j)^k,\theta_m\big\rangle\right|<1/j\;,\;\,1\leq m\leq j\,.$$
Since $R_j$, $c_1(L,h_j)^k$, have uniformly bounded mass and $c_1(L,h_j)^k\to c_1(L,h)^k$, it follows that $R_j\to c_1(L,h)^k$ weakly on $X$ as $j\to\infty$. The proof of Theorem \ref{T:approx} is complete.
\end{proof}

\begin{Remark}\label{R:Kahler} The hypothesis of Theorem \ref{T:approx} that $X$ carries a K\"ahler form $\Omega$ is essential. In Step 2 of the proof we apply \cite[Theorem 5.4]{CM11}, which requires $X$ to be a K\"ahler manifold. Moreover, we use that $\Omega$ is a K\"ahler form in formula \eqref{e:mass}, which shows that all currents $\frac{1}{p^k}\,\gamma_p^k$, $R_p\in\mathcal A_k(L^p)$, have equal mass with respect to $\Omega^{n-k}$.
\end{Remark}

\begin{proof}[Proof of Theorem \ref{T:approx2}] We note first that Theorem \ref{T:approx2} holds if $h$ is a smooth metric with positive curvature $c_1(L,h)>0$. Indeed, by \cite[Th\'eor\`eme 7.3]{DS06b} (see also \cite{ShZ08}) there exist sections $\sigma_{p,1},\dots,\sigma_{p,k}\in H^0(X,L^p)$ such that $[\sigma_{p,1}=0]\wedge\ldots\wedge[\sigma_{p,k}=0]$ are well defined positive closed currents of bidegree $(k,k)$ and $\frac{1}{p^k}\,[\sigma_{p,1}=0]\wedge\ldots\wedge[\sigma_{p,k}=0]\to c_1(L,h)^k$ weakly on $X$ as $p\to\infty$.

\par We consider next the general case of a singular metric $h$ with $c_1(L,h)\geq0$ and such that the current $c_1(L,h)^k$ is well defined for some $k\leq n$. Let $h_0$ be a smooth metric on $L$ such that $\Omega=c_1(L,h_0)$ is a K\"ahler form on $X$ and let $\varphi$ be the $\Omega$-psh function determined by $h$ (see \cite{D90,GZ05}). The regularization theorem \cite[Theorem 1]{BK07} (see also \cite{D92}, \cite[Theorem 8.1]{GZ05}) yields a decreasing sequence of smooth $\Omega$-psh functions $\varphi_j\searrow\varphi$ on $X$. Subtracting a constant, we may assume that $\varphi_1<0$ on $X$. Then $\psi_j:=\frac{j}{j+1}\,\varphi_j$ are smooth $\Omega$-psh functions on $X$ and $\psi_j\searrow\varphi$ as $j\to\infty$. If $h_j$ is the metric on $L$ defined by $\psi_j$ then $c_1(L,h_j)=\Omega+dd^c\psi_j\geq\frac{1}{j+1}\,\Omega>0$. Since the Monge-Amp\`ere operator is continuous on decreasing sequences it follows that $c_1(L,h_j)^k\to c_1(L,h)^k$ weakly on $X$.

\par We proceed now as in Step 3 from the proof of Theorem \ref{T:approx}. Let $\{\theta_m\}_{m\geq1}$ be a dense set of smooth $(n-k,n-k)$ forms on $X$. As noted at the beginning of the proof, we can apply the conclusion of Theorem \ref{T:approx2} for each smooth positively curved metric $h_j$. Thus we obtain a sequence of integers $p_j\nearrow\infty$ and sections $\sigma_{j,1},\dots,\sigma_{j,k}\in H^0(X,L^{p_j})$ such that $T_j:=[\sigma_{j,1}=0]\wedge\ldots\wedge[\sigma_{j,k}=0]$ are well defined positive closed currents of bidegree $(k,k)$ and
$$\left|\big\langle p_j^{-k}\,T_j,\theta_m\big\rangle-\langle c_1(L,h_j)^k,\theta_m\rangle\right|<1/j\;,\;\,1\leq m\leq j\,.$$
Since the currents $p_j^{-k}\,T_j$, $c_1(L,h_j)^k$, have uniformly bounded mass and $c_1(L,h_j)^k\to c_1(L,h)^k$, it follows that $p_j^{-k}\,T_j\to c_1(L,h)^k$ weakly on $X$ as $j\to\infty$. $\qed$

\medskip

\par\noindent{\em Proof of Theorem \ref{T:approx3}.} Since $L$ is big, it carries a singular Hermitian metric $h_0$ with strictly positive curvature current \cite{JS93} (see also \cite[Theorem 2.3.8]{MM07}). If $h$ is any such metric, \cite[Theorem 5.1]{CM11} shows that there exists a sequence of sections $\sigma_p\in H^0(X,L^p)$ such that $\frac{1}{p}\,[\sigma_p=0]\to c_1(L,h)$ weakly on $X$ as $p\to\infty$.

\par Assume now that $h$ is a singular Hermitian metric on $L$ with positive curvature current $c_1(L,h)\geq0$. Let $h_s$ be a fixed smooth metric on $L$, $\theta=c_1(L,h_s)$, and let $\varphi,\varphi_0$ be the $\theta$-psh functions determined by $h$, resp. $h_0$. If $h_j$ is the singular metric on $L$ determined by the $\theta$-psh function $\varphi_j=(j\varphi+\varphi_0)/(j+1)$ then its curvature is a K\"ahler current, since $c_1(L,h_j)\geq(j+1)^{-1}c_1(L,h_0)$. As $\varphi_j\to\varphi$ in $L^1(X)$, it follows that $c_1(L,h_j)\to c_1(L,h)$ weakly on $X$. Then we conclude the proof proceeding as at the end of the proof of Theorem \ref{T:approx2}, since Theorem \ref{T:approx3} holds for each metric $h_j$. 
\end{proof}

\begin{Example}
We conclude this section by discussing a fundamental example where Theorem \ref{T:approx2} applies. Consider the line bundle $L=\mathcal O(1)$ over $X={\mathbb P}^n$. The global holomorphic sections of $L^d$ are given by homogeneous polynomials of degree $d$ on ${\mathbb C}^{n+1}$. If $\omega$ denotes the Fubini-Study K\"ahler form on ${\mathbb P}^n$ then the set of singular metrics $h$ on $L$ is in one-to-one correspondence to the set of $\omega$-psh functions $\varphi$ on ${\mathbb P}^n$, and we have $c_1(L,h)=\omega+dd^c\varphi$ (\cite{D90}, \cite{GZ05}). Moreover, the latter class is in one-to-one correspondence to the Lelong class $\mathcal L({\mathbb C}^n)$ of entire psh functions with logarithmic growth, where we consider the standard embedding ${\mathbb C}^n\hookrightarrow{\mathbb P}^n$ (cf.\ \cite[Section 2]{GZ05}). Recall that a psh function $u$ on ${\mathbb C}^n$ is said to have logarithmic growth if $u(z)\leq\log^+\|z\|+C$ holds on ${\mathbb C}^n$, with some constant $C$ depending on $u$.

\par The domain $DMA({\mathbb P}^n,\omega)$ of the complex Monge-Amp\`ere operator is defined in \cite{CGZ08} to be the set of $\omega$-psh functions $\varphi$ on ${\mathbb P}^n$ for which there is a positive Radon measure $MA(\varphi)$ with the following property: If $\{\varphi_j\}$ is any sequence of {\em bounded} $\omega$-psh functions decreasing to $\varphi$ then $(\omega+dd^c {\varphi_j})^n\rightarrow MA(\varphi)$, in the weak sense of measures. We set $(\omega+dd^c\varphi)^n:=MA(\varphi)$. This domain includes all the classes in which the operator was defined earlier, either as a consequence of the local theory \cite{BT76,BT82,Ce04,Bl04,Bl06}, or genuinely in the compact setting (the class ${\mathcal E}$ from \cite{GZ07}).

\begin{Theorem}
For every $\varphi\in DMA({\mathbb P}^n,\omega)$ there exist a sequence of integers $d_j\nearrow\infty$ and $n$-tuples $(P_{j,1},\dots,P_{j,n})$ of homogeneous polynomials of degree $d_j$ on ${\mathbb C}^{n+1}$ such that for each $j$ the set $\{P_{j,1}=0\}\cap\ldots\cap\{P_{j,n}=0\}\subset{\mathbb P}^n$ is finite and the measures $d_j^{-n}\,[P_{j,1}=0]\wedge\ldots\wedge[P_{j,n}=0]$ converge weakly on ${\mathbb P}^n$ to $(\omega+dd^c\varphi)^n$.
\end{Theorem}

\par This follows immediately by applying Theorem \ref{T:approx2} to the singular metric on $L$ determined by $\varphi$.
\end{Example}

\section{Equidistribution for sections of adjoint bundles}\label{S:rmad}
We return here to a general framework analogous to \cite{CM11} and consider sections of adjoint bundles. The benefit is that the analysis of the Bergman kernel becomes easier, since its definition doesn't depend in this case on the ground metric and $L^2$ estimates do not involve the Ricci curvature. This is especially useful when dealing with singular spaces or metrics and was already done for orbifolds in \cite{CM12}. We consider the following setting:

\smallskip

(A) $X$ is a complex manifold of dimension $n$ (not necessarily compact), $\Sigma$ is a compact analytic subvariety of $X$, and $\Omega$ is a smooth positive $(1,1)$ form on $X\setminus\Sigma$.

\smallskip

(B) $(L,h)$ is a holomorphic line bundle on $X$ with a singular Hermitian metric $h$ with positive curvature current $c_1(L,h)\geq0$ on $X$.

\smallskip

\par Consider the space $H^0_{(2)}(X\setminus\Sigma,L^p\otimes K_X)$ of $L^2$-holomorphic sections of $L^p\otimes K_X\mid_{_{X\setminus\Sigma}}$ relative to the metrics $h_p$ on $L^p$ induced by $h$, $h^{K_{X}}$ on $K_X\mid_{_{X\setminus\Sigma}}$ induced by $\Omega$ and the volume form  $\Omega^n$ on $X\setminus\Sigma$, endowed with the inner product
$$(S,S^{\,\prime})_p=\int_{X\setminus\Sigma}\langle S,S^{\,\prime}\rangle_{h_p,\Omega}\,\Omega^n\,,\;\;S,S^{\,\prime}\in H^0_{(2)}(X\setminus\Sigma,L^p\otimes K_X).$$
The interesting point is that the space $H^0_{(2)}(X\setminus\Sigma,L^p\otimes K_X)$ does not depend on the choice of the form $\Omega$. Indeed, for any $(n,0)$-form $S$ with values in $L^p$, and any metrics $\Omega$, $\Omega_1$ on $X\setminus\Sigma$, we have pointwise $\vert S\vert_{h_p,\Omega}^2\Omega^n=\vert S\vert_{h_p,\Omega_1}^2\Omega_1^n$. Therefore, we can take $\Omega$ to be a smooth positive (1,1) form on $X$. Since $c_1(L,h)\geq0$, the metric $h$ has psh local weights, so it is locally bounded below away from zero. Thus sections $S\in H^0_{(2)}(X\setminus\Sigma,L^p\otimes K_X)$ are locally integrable on $X$ (with respect to the Lebesgue measure). Skoda's lemma \cite[Lemma\,2.3.22]{MM07} shows that sections in $H^0_{(2)}(X\setminus\Sigma,L^p\otimes K_X)$ extend holomorphically to $X$, therefore $H^0_{(2)}(X\setminus\Sigma,L^p\otimes K_X)\subset H^0(X,L^p\otimes K_X)$ and
\begin{eqnarray*}
H^0_{(2)}(X\setminus\Sigma,L^p\otimes K_X)&=&H^0_{(2)}(X,L^p\otimes K_X)\\
&=&\Big\{S\in H^0(X,L^p\otimes K_X):\int_{X}|S|^2_{h_p,\widetilde\Omega}\,\widetilde\Omega^n<\infty\Big\}\,,
\end{eqnarray*}
where $\widetilde\Omega$ is any smooth positive $(1,1)$ form on $X$.

\par Let $d_p:=\dim H^0_{(2)}(X,L^p\otimes K_X)\in\mathbb N\cup\{\infty\}$ and let $\{S^p_j\}_j$ be an orthonormal basis of $H^0_{(2)}(X,L^p\otimes K_X)$. Denote by $P_p$ the Bergman kernel function defined by
\begin{equation}\label{e:Bergfcn}
P_p(x)=\sum_{j=1}^{d_p}|S^p_j(x)|_{h_p,\Omega}^2\,,\;\;|S^p_j(x)|_{h_p,\Omega}^2:=\langle S_j^p(x),S_j^p(x)\rangle_{h_p,\Omega}\,,\;x\in X\,.
\end{equation}
We denote by $\gamma_p$ the Fubini-Study current of the space $H^0_{(2)}(X,L^p\otimes K_X)$, defined by
\begin{equation}\label{e:gammap1}
\gamma_p\mid_{_{U}}=\frac{1}{2}\,dd^c\log\Big(\sum_{j=1}^{d_p}|s_j^p|^2\Big),\;U\subset X\,\text{ open}\,,
\end{equation}
where $S_j^p=s_j^p\,e^{\otimes p}\otimes e'$, and $e$, $e'$ are local holomorphic frames for $L$, $K_X$ on $U$. As in \cite[Lemma 3.2]{CM11} (see also \cite[(3.48)]{MM08}) we see that the currents $\gamma_p$ are independent of the choice of basis $\{S_j^p\}$ and we have
$$\frac{1}{p}\,\gamma_p=c_1(L,h)+\frac{1}{p}\,c_1(K_X,h^{K_X})+\frac{1}{2p}\,dd^c\log P_p\,.$$
By taking $\Omega$ to be smooth on $X$, we have $\frac{1}{p}\,c_1(K_X,h^{K_X})\to0$ locally uniformly as $p\to\infty$.

\par We denote by $[S=0]$ the current of integration (with multiplicities) over the analytic hypersurface $\{S=0\}$ determined by a nontrivial section $S\in H^0(X,L^p\otimes K_X)$. The Lelong-Poincar\'e equation reads
$$\frac{1}{p}\,[S=0]=c_1(L,h)+\frac{1}{p}\,c_1(K_X,h^{K_X})+\frac{1}{2p}\,dd^c\log |S|_{h_p,\Omega}^2\,.$$
When $X$ is compact, let ${\mathcal S}^p$ be the unit sphere of $(H^0_{(2)}(X,L^p\otimes K_X),\|\cdot\|_{L^2})$ and let $\lambda_p$ be the normalized surface measure on ${\mathcal S}^p$. We denote by $\lambda_p^k$ the product measure on $({\mathcal S}^p)^k$. We also consider the probability space ${\mathcal S}_\infty=\prod_{p=1}^\infty{\mathcal S}^p$ endowed with the probability measure $\lambda_\infty=\prod_{p=1}^\infty\lambda_p$\,.

\begin{Theorem}
Let $X,\,\Sigma,\,\Omega,\,(L,h)$ satisfy (A), (B) and assume that there exists an open set $G\subset X$ such that $c_1(L,h)$ is strictly positive on $G\setminus\Sigma$, i.e. $c_1(L,h)\geq\varepsilon\,\Omega$ on $G\setminus\Sigma$, for some continuous function $\varepsilon:G\setminus\Sigma\longrightarrow(0,+\infty)$.
\\[2pt]
(i) Assume that $X\setminus\Sigma$ admits a complete K\"ahler metric. Then $\frac{1}{p}\,\gamma_p\to c_1(L,h)$ weakly as currents on $G$ and $\frac{1}{p}\,\log P_p\to 0$  in $L^1_{loc}(G,\Omega^n)$, as $p\to\infty$.
\\[2pt]
(ii) Assume that $X$ is a compact K\"ahler manifold. Then $\frac{1}{p}\,[\sigma_p=0]\to c_1(L,h)$ weakly on $G$ as $p\to\infty$, for $\;\lambda_\infty\text{-a.e. sequence }\{\sigma_p\}_{p\geq1}\in{\mathcal S}_\infty$.
\end{Theorem}

\begin{proof}
To prove $(i)$, we repeat the proof of \cite[Theorem\,1.1]{CM12}, by replacing the orbifold regular locus $X^{orb}_{reg}$ in loc.\ cit.\ by $X\setminus\Sigma$\,. The proof of $(ii)$ is analogous to the proof of \cite[Theorem\,1.6]{CM12}.
\end{proof}

\par Consider now the following condition:

\smallskip

(B$^\prime$) $(L,h)$ is a holomorphic line bundle on $X$ with a singular Hermitian metric $h$ on $X$ such that $h$ is \emph{continuous} on $X\setminus\Sigma$ and $c_1(L,h)\geq0$ on $X$.

\smallskip

\par Proceeding as in the proof of Theorem \cite[Theorem\,1.1]{CM11} we obtain the following:

\begin{Theorem}\label{T:mt}
Let  $X,\,\Sigma,\,\Omega,\,(L,h)$ satisfy (A), (B\,$^\prime$) and assume that \begin{equation}\label{e:mainhyp}
\lim_{p\to\infty}\frac{1}{p}\,\log P_p(x)=0, \text{ locally uniformly on } X\setminus\Sigma\,.
\end{equation}
Then $\frac{1}{p}\,\gamma_p\to c_1(L,h)$ weakly on $X$. If, in addition, $\dim\Sigma\leq n-k$ for some $2\leq k\leq n$, then the currents $\gamma^k$ and $\gamma_p^k$ are well defined on $X$, respectively on each relatively compact neighborhood of $\Sigma$, for all $p$ sufficiently large. Moreover, $\frac{1}{p^k}\,\gamma_p^k\to\gamma^k$ weakly on $X$.
\end{Theorem}

\par By slight modifications of the proofs of \cite[Theorems\,1.2,\,4.3]{CM11} we obtain:

\begin{Theorem}\label{T:SZ} Let  $X,\,\Sigma,\,\Omega,\,(L,h)$ satisfy (A), (B\,$^\prime$). Assume that $X$ is compact, $\dim\Sigma\leq n-k$ for some $1\leq k\leq n$, and that \eqref{e:mainhyp} holds. Then, for all $p$ sufficiently large:

\par (i) $[\sigma=0]:=[\sigma_1=0]\wedge\ldots\wedge[\sigma_k=0]$ is a well defined positive closed current of bidegree (k,k) on $X$, for $\lambda_p^k$-a.e. $\sigma=(\sigma_1,\dots,\sigma_k)\in({\mathcal S}^p)^k$.

\par (ii) The expectation $E_p^k[\sigma=0]$ of the current-valued random variable $\sigma\to[\sigma=0]$, given by
$\langle E_p^k[\sigma=0],\varphi\rangle=\int_{({\mathcal S}^p)^k}\langle[\sigma=0],\varphi\rangle\,d\lambda_p^k$,
where $\varphi$ is a test form on $X$, is a well defined current and $E_p^k[\sigma=0]=\gamma_p^k$.

\par (iii) We have $\frac{1}{p^k}\,E_p^k[\sigma=0]\to\gamma^k$ as $p\to\infty$, weakly in the sense of currents on $X$.
\end{Theorem}

\begin{Theorem}\label{T:SZ1} Let  $X,\,\Sigma,\,\Omega,\,(L,h)$ satisfy (A), (B\,$^\prime$) such that $X$ is compact and \eqref{e:mainhyp} holds. Then $\frac{1}{p}\,[\sigma_p=0]\to c_1(L,h)$ as $p\to\infty$ weakly on $X$, for $\;\lambda_\infty\text{-a.e. sequence }\{\sigma_p\}_{p\geq1}\in{\mathcal S}_\infty$.
\end{Theorem}

\par Taking advantage of the fact that we work with adjoint bundles, we apply in the next result the analysis already used in \cite{CM12} (especially the resolution of the $\overline\partial$-equation on complete K\"ahler manifolds, cf.\ \cite[\S4.2-3]{CM12}).

\begin{Theorem}\label{T:mt2}
Let  $X,\,\Sigma,\,\Omega,\,(L,h)$ satisfy (A), (B\,$^\prime$) and let $k\leq\codim\Sigma$. Assume that $X\setminus\Sigma$ admits a complete K\"ahler metric and $c_1(L,h)$ is strictly positive on $X\setminus\Sigma$, i.e. $c_1(L,h)\geq\varepsilon\,\Omega$ on $X\setminus\Sigma$, for some continuous function $\varepsilon:X\setminus\Sigma\longrightarrow(0,+\infty)$. Then \eqref{e:mainhyp} holds and the conclusions of Theorem \ref{T:mt} hold. If, in addition, $X$ is compact, then $\frac{1}{p}\,[\sigma_p=0]\to c_1(L,h)$ as $p\to\infty$ weakly on $X$, for $\lambda_\infty\text{-a.e. sequence }\{\sigma_p\}_{p\geq1}\in{\mathcal S}_\infty$. Moreover, the conclusions of Theorem \ref{T:SZ} hold.
\end{Theorem}

\begin{proof}
The locally uniform convergence \eqref{e:mainhyp} follows by showing the estimate \cite[(12)]{CM12} as in \cite[Theorem\,4.7]{CM12}, whereby $X^{orb}_{reg}$ in loc.\ cit.\ is replaced by $X\setminus\Sigma$\,. Hence we can apply Theorems \ref{T:SZ} and \ref{T:SZ1} in this situation.
\end{proof}

\begin{proof}[Proof of Theorem \ref{T:approx4}]
We repeat the proof of Theorem \ref{T:approx}, working now with the spaces $H^0_{(2)}(X,L^p\otimes K_X)$ of $L^2$-holomorphic sections of adjoint bundles, and using Theorem \ref{T:mt2} instead of \cite[Theorem 5.4]{CM11} in Step 2 of the proof. Note that $X\setminus\Sigma$ admits a complete K\"ahler metric since $X$ is a compact K\"ahler manifold (cf.\ \cite{O87}; see also \cite[Lemma\,4.9]{CM12}).
\end{proof}
 
\par Along the same lines as above, there are versions of \cite[Theorems\,6.19,\,6.21]{CM11} for $n$-forms.

\begin{Theorem}\label{T:1conv}
Let $X$ be a $1$-convex manifold and $\Sigma\subset X$ be the exceptional set of $X$. Let $(L,h)$ be a holomorphic line bundle on $X$ with singular metric $h$ such that $c_1(L,h)\geq0$ on $X$, $c_1(L,h)$ is strictly positive on $X\setminus\Sigma$, and $h\mid_{X\setminus\Sigma}$ is continuous. Then \eqref{e:mainhyp} holds. In particular, the conclusions of Theorem \ref{T:mt} hold for the spaces $H^0_{(2)}(X,L^p\otimes K_X)$.
\end{Theorem}

\begin{proof}
Indeed, $X\setminus\Sigma$ admits a complete
K\"ahler metric, so
this is an immediate application of Theorem \ref{T:mt2}.
\end{proof}

\begin{Remark}
If we suppose only that $c_1(L,h)$ is strictly positive in a neighborhood of $\Sigma$, we can modify the metric $h$ by $he^{-A\rho}$, where $\rho$ is the psh exhaustion function of $X$ and $A$ is an appropriate constant, so that the curvature of the line bundle $(L,he^{-A\rho})$ becomes positive on the whole $X\setminus\Sigma$. Thus $\frac{1}{p}\,\gamma_p\to c_1(L,he^{-A\rho})=c_1(L,h)+A\,dd^c\rho$ as $p\to\infty$ on $X$.
\end{Remark}

\end{document}